\documentclass[12pt,a4paper]{amsart}
\usepackage{graphics}
\usepackage{epsfig}
\usepackage{graphicx}
\theoremstyle{plain}
\usepackage{amssymb}
\usepackage{color}
\advance\hoffset-20mm \advance\textwidth40mm

\newtheorem{theorem}{Theorem}
\newtheorem{lemma}{Lemma}
\newtheorem*{theo*}{Theorem}
\newtheorem{proposition}{Proposition}
\newtheorem{corollary}{Corollary}
\theoremstyle{definition}

\newtheorem*{definition*}{Definition}

\def \K {\mathbb{K}}
\def \sl {\mathfrak{sl}}
\def \gl {\mathfrak{gl}}
\def \aff {\mathfrak{aff}}
\def \g {\mathfrak{g}}

\DeclareMathOperator{\Div}{div}

\begin{document}
\sloppy
\title[Module structure of the Lie algebra $W_n(K)$ over $sl_n(K)$]
{Module structure of the Lie algebra $W_n(K)$ over $sl_n(K)$}
\author
{Y.Chapovskyi, A.Petravchuk}
\address{Institute of Mathematics, National Academy of Sciences of Ukraine,
Tereschenkivska street, 3, 01004 Kyiv, Ukraine}
\email{safemacc@gmail.com}
\address{ Faculty of Mechanics and Mathematics,
Taras Shevchenko National University of Kyiv, 64, Volodymyrska street, 01033  Kyiv, Ukraine}
\email{ petravchuk@knu.ua, apetrav@gmail.com}

\date{\today}

\begin{abstract}
Let $\mathbb K$ be an algebraically closed field of characteristic zero, 
$A = \mathbb K[x_1,\dots,x_n]$ the polynomial ring,
and let $W_n(\mathbb K)$ denote the Lie algebra of all $\mathbb K$-derivations on $A$.
The Lie algebra $W_n := W_n(\mathbb K)$ admits a natural grading $W_n = \bigoplus_{i \ge -1} W^{[i]}_n$,
where $W^{[i]}_n$ consists of all homogeneous derivations whose coefficients are homogeneous polynomials of degree $i+1$ or zero. The component $W^{[0]}_n$
is a subalgebra of $W_n$ and is isomorphic to  $\mathfrak{gl}_n(\mathbb K).$ Moreover, each $W_n^{[i]}$ for $i \ge -1$ is a finite-dimensional module over $W_n^{[0]}$.
We prove that $W^{[i]}_n,\; i \ge 0$ is a sum of two irreducible submodules $W^{[i]}_n = M_i \oplus N_i$,
where $M_i$ consists of all  divergence-free derivations, and $N_i$ consists of derivations that are polynomial  multiples of the
Euler derivation $E_n = \sum_{i=1}^n x_i \frac{\partial}{\partial x_i}$.
As a consequence, we show that the standard grading is exact in certain sense, namely: 
$[W^{[i]}_n, W^{[j]}_n] = W^{[i+j]}_n$ for all $i,j,$ except when $i = j = 0$. 
We also address the question of when the subalgebra of $W_n$ generated by $W_n^{[-1]} \oplus W_n^{[0]},$ together with an additional element from $W_n,$ equals the entire Lie algebra $W_n$.

 \end{abstract}
\maketitle


\section{Introduction}

Let $\K$ be an algebraically closed field of characteristic zero, 
$A = \K[x_1,\dots,x_n]$ the polynomial ring,
and let $W_n := W_n(\K)$ be the Lie algebra of all the $\K$-derivations on $A$.
From the viewpoint of geometry, $W_n$ consists of all the polynomial vector fields on $\K^n$.
The Lie algebra $W_n$ was studied by many authors from different  viewpoints: maximal subalgebras of $W_n$  were  studied in \cite{Rudak1, Bavula, Bell}, solvable and nilpotent subalgebras in \cite{MP},   automorphisms and derivations of $ \K[x_1,\dots,x_n]$  and related structures   in \cite{Bavula, Bezushchak, Rudak2}. 
The Lie  algebra $W_n$  has a natural grading: $W_n = \bigoplus_{j \ge -1} W^{[j]}_n$,
where 
$$W^{[j]}_n = \{D \in W_n \mid D = 
\sum_{i=1}^n f_i \frac{\partial}{\partial x_i}, 
f_i\in \mathbb K[x_1, \ldots , x_n]\}$$ 
and  $f_i, i=1, \ldots , n$ are either homogeneous polynomials  of degree ${j+1}$ or zero.
If $D\in W^{[j]}_n$ then we say that $D$ is homogeneous of degree $\deg D=j. $
The $\mathbb K$-subspace  $W^{[0]}_n$ consists of all the linear homogeneous derivations and $W^{[0]}_n \simeq \gl_n(\K)$.
This is a subalgebra of $W_n$ and $W^{[0]}_n$ contains the subalgebra consisting of derivations  with zero divergence that is isomorphic to the simple Lie algebra  $\sl_n(\K)$. 
All the homogeneous components $W^{[i]}_n$ are modules over $W^{[0]}_n$ (and over $\sl_n(\K)$).
The main result of the paper is Theorem \ref{Th1}: every $W^{[0]}_n$-module $W^{[i]}_n, i\geq 0$ is a direct sum
$W^{[i]}_n = M_i \oplus N_i$ of two irreducible submodules $M_i$ and $N_i$.
The submodule $M_i$ consists of all the divergence-free derivations and $N_i$ consists  of all derivations that are  polynomial multiple of
the Euler derivation $E_n = \sum_{i=1}^n x_i \frac{\partial}{\partial x_i}$.
As a consequence of Theorem \ref{Th1}, we obtain the equalities $[W^{[i]}_n, W^{[j]}_n] = W^{[i+j]}_n,$
except when $i = j = 0$ (for the case $ i=j=-1$ this equality also holds if we put $W_n^{[-2]}=0 $).
As an application of Theorem \ref{Th1} we point out a criterion under which an element of $W_n$ generates together with
the (finite dimensional) subalgebra $W^{[-1]}_n \oplus W^{[0]}_n $ the entire Lie algebra $W_n$ (Theorem \ref{Th2}).

Notations in the paper are standard.
The subalgebra $W^{[0]}_n$ consists of all derivations 
$D = \sum_{i,j=1}^n a_{ij} x_i \frac{\partial}{\partial x_j}$
and is isomorphic to the general linear algebra $\gl_n(\K)$ via the map $D\to (a_{ij}).$  
Divergence-free derivations from $W^{[0]}_n$ form a subalgebra, 
that is isomorphic to the simple Lie algebra $\sl_n(\K)$.
Without loss of generality, one can assume that $\gl_n(\K)$ and $\sl_n(\K)$ are embedded in $W_n$.
For any $D \in W_n,$ $ D=\sum_{i=1}^n f_i \frac{\partial}{\partial x_i}, $ the divergence $\Div D$ is the sum 
$\Div D = \sum_{i=1}^n \frac{\partial f_i}{\partial x_i}$.

\section{The main result}
The next two statements can be checked immediately (see also \cite{Now}).

\begin{lemma} \label{lm:l1}
	Let $D_1, D_2 \in W_n, \quad f \in A = \K[x_1,\dots,x_n]$.
	Then the equalities hold:
	\begin{enumerate}
		\item $\Div(D_1 \pm D_2) = \Div D_1 \pm \Div D_2$;
		\item $\Div(f D_1) = f \Div D_1 + D_1(f)$;
		\item $\Div([D_1, D_2]) = D_1(\Div D_2) - D_2(\Div D_1)$.
	\end{enumerate}
\end{lemma}

\begin{lemma} \label{lm:l2}
	Let $E_n = \sum_{i=1}^n x_i \frac{\partial}{\partial x_i}$ be the Euler derivation,
	$D \in W^{[i]}_n$ and let $f \in A$ be a homogeneous polynomial of degree $m$.
	Then
	\begin{enumerate}
		\item $E_n(f) = mf;$
		\item $ [E_n, D] = iD,$
		in particular, for $D \in W^{[0]}_n$ it holds $[E_n, D] = 0;$
		\item $\Div(f E_n) = (m+n)f.$
	\end{enumerate}
\end{lemma}

In order to prove the main result of the paper, Theorem \ref{Th1}, we need some results from representation theory of semisimple Lie algebras.
Let $\g$ be a semisimple Lie algebra over an algebraically closed  field $\mathbb K$ 
 of characteristic zero. 
Let $H$ be a Cartan subalgebra of $\g$, let 
$\varPhi$ be its root system and 
$\Delta$  the basis of $\varPhi$.
Denote by $B(\Delta) = H \oplus \bigoplus_{\alpha \in \varPhi , \alpha \succ 0} \g_{\alpha}$ the Borel subalgebra
of the Lie algebra $\g$. In the following lemma, we gather some eesential facts  
 about finite-dimensional modules over the Lie algebra $\g$.

\begin{lemma}[see, for example \cite{Humphreys}, Ch.VI] \label{lm:l3}
	Let $V$ be a finite-dimensional module over a semisimple Lie algebra $\g$.
	Then:
	\begin{enumerate}
		\item The module $V$ contains at least one maximal eigenvector of the highest weight of $V$, i.e.,
		 a vector $v=v_{\lambda}$  such that 
		$x_{\alpha} v = 0$ for all $x_{\alpha} \in \g_{\alpha},\; \alpha \in \varPhi, \; \alpha \succ 0$ and $hv = \lambda(h)v$ for all $h \in H$.
		\item If the module $V$ contains a unique maximal vector 
		(up to nonzero scalar multiples) then $V$ is an irreducible module.
			\item If $U, V$ are irreducible non-isomorphic $\g$-modules, then the module $U\oplus V$ has only three proper submodules:   $U, V$ and $0$.
		
	\end{enumerate}
\end{lemma}

\begin{lemma} \label{lm:l4}
	Let $D \in W_n^{[m]}$, $m \ge 0$ be a homogeneous (of degree $m$) derivation such that 
	$\left[x_{\alpha} \frac{\partial}{\partial x_{\beta}}, D \right] = 0$ for all $1 \le \alpha < \beta \le n$.
	Then $D$ is of the form
	$$
	D = c^1 x_1^{m+1} \frac{\partial}{\partial x_1} + 
	\sum_{i=2}^n \left(\sum_{k=0}^{m+1} c_k^i x_1^{m-k+1} x_i^k \right) \frac{\partial}{\partial x_i}
	$$
	for some $c^1,\, c_k^i \in \K$, $2 \le i \le n$, $0 \le k \le m+1$.
\end{lemma}
\begin{proof}
	Let us write  the derivation $D$ in the coordinate form:
	\begin{equation} \label{eq:e1}
		D = \sum_{i=1}^n \left( \sum_{|J| = m+1} a_J^i \ x_1^{j_1} \dots x_n^{j_n} \right) \frac{\partial}{\partial x_i},
	\end{equation}
	where $J = (j_1,\dots,j_n)$, $j_s \ge 0$ and $|J| = j_1 + \dots + j_n$.
	Note that the next equality holds:
	\begin{equation} \label{eq:e2}
		\left[ x_{\alpha} \frac{\partial}{\partial x_{\beta}}, D\right] = 
		x_{\alpha} \left[\frac{\partial}{\partial x_{\beta}}, D\right] -
		D(x_{\alpha}) \frac{\partial}{\partial x_{\beta}}.
	\end{equation}
	Then we have  
	$$
	\left[\frac{\partial}{\partial x_{\beta}}, D\right] = \sum_{i=1}^n 
	\left(\sum_{|J| = m+1} j_{\beta} a_J^i \ x_1^{j_1} \dots x_{\beta}^{j_{\beta}-1} \dots x_n^{j_n}\right)
	\frac{\partial}{\partial x_i},
	$$
	$$
	D(x_{\alpha}) = \sum_{|J| = m+1} a_J^{\alpha} \ x_1^{j_1} \dots x_n^{j_n}.
	$$
	Replacing in \eqref{eq:e2} $\left[\frac{\partial}{\partial x_{\beta}}, D\right]$ and $D(x_{\alpha})$
	by the above obtained expressions we get
		\begin{equation}\label{three}
	\left[ x_{\alpha} \frac{\partial}{\partial x_{\beta}}, D\right] = \sum_{i=1}^n \left(\sum_{|J| = m+1} j_{\beta} a_J^i \ x_{\alpha} x_1^{j_1} \dots x_{\beta}^{j_{\beta}-1} \dots x_n^{j_n}\right)
	\frac{\partial}{\partial x_i} - 
		\end{equation}
	$$	
	-\left(\sum_{|J| = m+1} a_J^{\alpha} \ x_1^{j_1} \dots x_n^{j_n}\right)
	\frac{\partial}{\partial x_{\beta}} = 0
	$$
	for all $\alpha$, $\beta$, $1 \le \alpha < \beta \le n$.
	Rewrite this equality  in the form
	$$\left[ x_{\alpha} \frac{\partial}{\partial x_{\beta}}, D\right] = 
	\sum_{i=1}^n f_i \frac{\partial}{\partial x_i},  \ f_i \in \K[x_1,\dots,x_n],$$
	where 
	$$f_i = \sum_{|J| = m+1} 
	j_{\beta} a_J^i \ x_{\alpha} x_1^{j_1} \dots x_{\beta}^{j_{\beta}-1} \dots x_n^{j_n}$$
	for $i \ne \beta$ and $$f_{\beta} =\sum_{|J| = m+1} j_{\beta} a_J^i \ x_{\alpha} x_1^{j_1} \dots x_{\beta}^{j_{\beta}-1} \dots x_n^{j_n} - \sum_{|J| = m+1} a_J^{\alpha} \ x_1^{j_1} \dots x_n^{j_n}.$$
	If $i \ne \beta$ then \eqref{three} implies the equality $f_i = 0$
	because $\frac{\partial}{\partial x_i} \ne \frac{\partial}{\partial x_{\beta}}$.
	If, in addition, $j_{\beta} \ne 0$ then $a_J^i = 0$ in the expression for $f_i$.
	Thus, for $i \ne \beta$ and $j_{\beta} \ne 0$ we have $a_J^i = 0$.
	Note that \eqref{three} is a collection of $\frac{n(n-1)}{2}$ equalities for
	$\alpha, \beta \in \{1,2,\dots,n\}$, $\alpha < \beta$.
	
	First, consider the coefficient
	\begin{equation} \label{eq:e4}
		f_1 = \sum_{|J| = m+1} j_{\beta} a_J^1 \ x_{\alpha} x_1^{j_1} \dots x_{\beta}^{j_{\beta}-1} \dots x_n^{j_n}
	\end{equation}
	at $\frac{\partial}{\partial x_1}$ in \eqref{three}. Here $i=1$ and $\beta \ge 2,$ so  we have
	$\beta \ne i$. If, in addition, $j_{\beta} \ne 0$ for some $\beta \ge 2$ then by the mentioned above
	$a_J^1 = 0$. Therefore  the polynomial $\sum_{|J| = m+1} a_J^1 \ x_1^{j_1} \dots x_n^{j_n}$
	contains only summands with $j_{\beta} = 0$, $\beta \ge 2$. The latter means that
	$$\sum_{|J| = m+1} a_J^1 \ x_1^{j_1} \dots x_n^{j_n} = c^1 x_1^{m+1}$$ for some $c^1 \in \K$ and
	\begin{equation} \label{eq:e5}
		D = c^1 x_1^{m+1} \frac{\partial}{\partial x_1} + \sum_{i=2}^n
		\left( \sum_{|J| = m+1} a_J^i \ x_1^{j_1} \dots x_n^{j_n} \right) \frac{\partial}{\partial x_i}.
	\end{equation}
	Let us find the coefficient 
	$$f_s = \sum_{|J| = m+1} a_J^s \ x_1^{j_1} \dots x_n^{j_n}$$
	at $\frac{\partial}{\partial x_s}$ in \eqref{eq:e5} for $2 \le s \le n$.
	Let us show that $f_s$ contains only monomials of $x_1$, $x_s$.
	Indeed, let it  not be the case. Then there exists a nonzero monomial
	$a_J^s x_1^{j_1} \dots x_n^{j_n}$ such that $j_{\beta} \ne 0$ for $\beta \ne s$.
	Then, by the above mentioned, $a_J^s = 0$, where $a_J^s$ is the coefficient at  the monomial
	from $f_s$ containing $x_{\beta}^{j_{\beta}}$ with $j_{\beta} \ne 0$.
	Therefore, $f_s$ contains only monomials of $x_1$, $x_s,$ i.e., 
	$f_s = \sum_{k=0}^{m+1} c_k^s x_1^{m-k+1} x_s^k$.
	So, we have proved that $D$ is of the required  form.
	The proof is complete.
\end{proof}

\begin{theorem}\label{Th1}
	Let the Lie algebra $W_n(\K) = W_n, n\geq 2$ be written as a direct sum of homogeneous 
	components of the standard grading
	\begin{equation} \label{eq:e6}
		W_n = W_n^{[-1]} \oplus W_n^{[0]} \oplus \dots \oplus W_n^{[m]} \oplus \dots.
	\end{equation}
	Then $L = W_n^{[0]}$ is a subalgebra of \  $W_n$, $L \simeq \gl_n(\K)$
	and every summand of the sum \eqref{eq:e6} is a finite dimensional module over $L$.
	 Every $L$-module $W_n^{[m]}, m \ge 0$ is a direct sum $W_n^{[m]} = M_m \oplus N_m$
	of two irreducible submodules, where $M_m$ consists of divergence-free derivations and $N_m$
	consists of all the derivations from $W_n^{[m]}$ that are polynomial multiple of the Euler derivation $E_n$.
\end{theorem}
\begin{proof}
	One can easily show (using Lemmas \ref{lm:l1} and \ref{lm:l2}) that $M_m$ and $N_m$ are $L$-submodules
	of the $L$-module $W_n^{[m]}$, $m \ge 0$.
	Recall that the subalgebra $L = W_n^{[0]}$ consists of linear derivations of the form
	$$D = \sum_{i,j=1}^n a_{ij}x_i \frac{\partial}{\partial x_j}, a_{ij} \in \K$$ and 
	$L$ is isomorphic to the general linear Lie algebra $\gl_n(\K)$.
	The subalgebra $M_0$ of $L$, consisting of divergence-free derivations is isomorphic to the simple
	subalgebra $\sl_n(\K) \subset \gl_n(\K)$ because $\Div D = {\rm tr} (a_{ij}) = 0$.
	Obviously, $W_n^{[0]} = L = M_0 \oplus N_0$, where $N_0 = \K E_n$.
	Since the Euler derivation $E_n$ acts scalarly on $W_n^{[m]}$ (see Lemma \ref{lm:l2}) 
	one  can consider $W_n^{[m]}$  (without loss of generality) as a module over $\sl_n(\K)$. 
	Choose the Borel subalgebra of $\sl_n(\K)$ consisting of all upper triangular matrices. 	
	By Lemma \ref{lm:l3}, it is sufficient to prove that the $\sl_n(\K)$-module $W_n^{[m]}$
	has exactly two linearly independent maximal vectors.
	By Lemma \ref{lm:l4}, any maximal vector of the $\sl_n(\K)$-module $W_n^{[m]}$, $m \ge 0$
	is of the form
	\begin{equation} \label{eq:e7}
		D = c^1 x_1^{m+1} \frac{\partial}{\partial x_1} +
		\sum_{i=2}^n \left(\sum_{k=0}^{m+1} c_k^i x_1^{m-k+1} x_i^k\right)\frac{\partial}{\partial x_i},
	\end{equation}
	for some $c^1, c_k^i \in \K$.
	
	Let us simplify the expression for $D$ using the equalities  
	$\left[x_{\alpha}\frac{\partial}{\partial x_{\beta}}, D\right] = 0$
	for all $1 \le \alpha < \beta \le n$.
	Recall that we assume that the Lie algebra $\sl_n(\K)$ is embedded in $W_n^{[0]}$ as a subalgebra consisting of all divergence-free derivations and  the Borel subalgebra of $\sl_n(\K)$ consists of all upper triangular matrices, i.e., of linear combinations (over $\mathbb K $) of derivations 
	$$ x_{\alpha}\frac{\partial}{\partial x_{\beta}},  \  \ x_{\alpha} \frac{\partial}{\partial x_{\alpha}} - x_n \frac{\partial}{\partial x_n}, 1\leq \alpha <\beta \leq n.$$
	As in the proof of Lemma \ref{lm:l4} we will use the equality \eqref{eq:e2}, namely
	$$
	\left[ x_{\alpha} \frac{\partial}{\partial x_{\beta}}, D\right] = 
	x_{\alpha} \left[\frac{\partial}{\partial x_{\beta}}, D\right] -
	D(x_{\alpha}) \frac{\partial}{\partial x_{\beta}}.
	$$
	Take  the derivation $D$ from \eqref{eq:e7}. Then  the product $\left[\frac{\partial}{\partial x_{\beta}}, D\right] $ is of the form
	$$
	\left[\frac{\partial}{\partial x_{\beta}}, D\right] = 
	\left(\sum_{k=0}^{m+1} k c_k^{\beta} \ x_1^{m-k+1} x_{\beta}^{k-1}\right)
	\frac{\partial}{\partial x_{\beta}}
	$$
	(recall that $\beta \ge 2$).
	Further,
	$$
	D(x_{\alpha}) = \delta_{1 \alpha} c^1 x_1^{m+1} + 
	\left(\sum_{k=0}^{m+1} c_k^{\alpha} x_1^{m-k+1} x_{\alpha}^k\right) (1 - \delta_{1 \alpha})
	$$
	(here, the multipliers $\delta_{1 \alpha}$ and $(1 - \delta_{1 \alpha})$ are included in the formula
	to account for the fact that $D$ acts differenly on $x_1$  than it does on $x_{\alpha}$, $\alpha > 1$).
	By the equality \eqref{eq:e2}, we have
	\begin{align*}
		&\left[ x_{\alpha} \frac{\partial}{\partial x_{\beta}}, D\right] =
		\left(\sum_{k=0}^{m+1} k c_k^{\beta} \ x_1^{m-k+1} x_{\beta}^{k-1} x_{\alpha} \right)
		\frac{\partial}{\partial x_{\beta}} \\-&
		\left(
		\delta_{1 \alpha} c^1 x_1^{m+1} + 
		(1 - \delta_{1 \alpha}) \sum_{k=0}^{m+1} c_k^{\alpha} x_1^{m-k+1} x_{\alpha}^k 
		\right) \frac{\partial}{\partial x_{\beta}} = 0
	\end{align*}
	for all $1 \le \alpha < \beta \le n$. 
	The latter equality implies
	\begin{equation} \label{eq:e8}
		\sum_{k=0}^{m+1} k c_k^{\beta} \ x_1^{m-k+1} x_{\beta}^{k-1} x_{\alpha} -
		\delta_{1 \alpha} c^1 x_1^{m+1} -
		(1 - \delta_{1 \alpha}) \sum_{k=0}^{m+1} c_k^{\alpha} x_1^{m-k+1} x_{\alpha}^k  = 0.
	\end{equation}
	Consider the monomials from the left side of \eqref{eq:e8} with $k \ge 2$
	 (recall that $\alpha < \beta$).
	The monomials from the first sum in \eqref{eq:e8} with $k \ge 2$ cannot be canceled with other	monomials,  so we have $c_k^{\beta} = 0$, $k \ge 2$. 
	Therefore the equality \eqref{eq:e8} takes the form
	\begin{equation} \label{eq:e9}
		c_1^{\beta} \ x_1^{m} x_{\alpha} -
		\delta_{1 \alpha} c^1 x_1^{m+1} -
		(1 - \delta_{1 \alpha}) \sum_{k=0}^{m+1} c_k^{\alpha} x_1^{m-k+1} x_{\alpha}^k  = 0,	 1 \le \alpha < \beta \le n.
	\end{equation}

	Further, consider the equalities \eqref{eq:e9} for $\alpha=1$, $\beta > \alpha$.
	One can easily see that $c_1^{\beta} = c^1$ for $2 \le \beta \le n$.
	Let now $2 \le \alpha \le n-1$ in \eqref{eq:e9}. Considering monomials
	with $k=0$ in  \eqref{eq:e9} yields the equalities
	$c_0^{\alpha} = 0$ for $2 \le \alpha \le n-1$. Thus,  we have in the equalities \eqref{eq:e9}: 
	
	$c_k^{\beta} = 0$, $2 \le k \le m+1$ for arbitrary $2 \le \beta \le n$; \ \ 
	 $c_0^{\alpha} = 0$	for $2 \le \alpha < n$;
	
	 $c_1^{\beta} = c^1$ for $2 \le \beta \le n$. The latter means that
	the derivation $D$ takes the form
	$$
	D = c^1 x_1^{m+1} \frac{\partial}{\partial x_1} + 
	\sum_{i=2}^n x_1^m x_i \frac{\partial}{\partial x_i} +
	c_0^n x_1^{m+1} \frac{\partial}{\partial x_n}
	.$$
	But then the derivation $D$ can be written in the form
	$D = c_1 x_1^m E_n + c_0^n x_1^{m+1} \frac{\partial}{\partial x_n}$.
	
	Let us check now that $D_{\lambda_1} = x_1^m E_n$ and
	$D_{\lambda_2} = x_1^{m+1} \frac{\partial}{\partial x_n}$ are maximal vectors
	for the Borel subalgebra of $\sl_n(\K)$ consisting of all upper triangular matrices.
	Indeed, 
	$$
	\left[x_{\alpha} \frac{\partial}{\partial x_{\beta}}, x_1^m E_n\right] =
	x_{\alpha} \frac{\partial}{\partial x_{\beta}}(x_1^m)E_n +
	x_1^m \left[x_{\alpha} \frac{\partial}{\partial x_{\beta}}, E_n\right] = 0,
	$$
	because $\left[x_{\alpha} \frac{\partial}{\partial x_{\beta}}, E_n\right] = 0$
	by Lemma \ref{lm:l2} and $\frac{\partial}{\partial x_{\beta}}(x_1^m) = 0$
	(recall that $\beta > 1$).
	Analogously one can check that 
	$\left[x_{\alpha} \frac{\partial}{\partial x_{\beta}}, x_1^{m+1} \frac{\partial}{\partial x_n}\right] = 0$
	(here we must take into account that $\alpha < n$).
	
	In addition, for all $1 \le \alpha < n$ we have
	\begin{align*}
		&\left[
		x_{\alpha} \frac{\partial}{\partial x_{\alpha}} - x_n \frac{\partial}{\partial x_n}, \
		x_1^m E_n
		\right] = \delta_{\alpha 1}m \ x_1^m E_n, \\
		&\left[
		x_{\alpha} \frac{\partial}{\partial x_{\alpha}} - x_n \frac{\partial}{\partial x_n}, \
		x_1^{m+1}\frac{\partial}{\partial x_n}
		\right] = (1 + \delta_{\alpha 1}(m+1)) \ x_1^{m+1} \frac{\partial}{\partial x_n}.
	\end{align*}
	The last two equalities (together with the two previous equalities) show that
	$D_{\lambda_1}$ and $D_{\lambda_2}$ are maximal vectors of the weights 
	$\lambda^{(1)}$ and $\lambda^{(2)}$ respectively, where
	\begin{align*}
		&\lambda^{(1)}(x_{\alpha} \frac{\partial}{\partial x_{\alpha}} - x_n \frac{\partial}{\partial x_n}) =
		m\delta_{\alpha 1}, \\
		&\lambda^{(2)}(x_{\alpha} \frac{\partial}{\partial x_{\alpha}} - x_n \frac{\partial}{\partial x_n}) =
		1 + (m+1)\delta_{\alpha 1}, \; 1 \le \alpha \le n-1.\\
	\end{align*}
	Note also that $D_{\lambda_1} \in N_m$ and $D_{\lambda_2} \in M_m$.
	The latter means (taking into account Lemma \ref{lm:l3}) that $M_m$ and $N_m$
	are irreducible submodules and $M_m + N_m = W_n^{[m]}$, $m \ge 0$.
	The proof is complete.
\end{proof}

\section{Some applications of Theorem 1}
\begin{proposition}\label{products}
Let $W_n = W_{n}^{[-1]}\oplus W_{n}^{[0]}\oplus\dots\oplus W_{n}^{[i]}\oplus\dots$ be the decomposition of $ W_n $ into  direct sum of homogeneous components of the standard grading on $W_{n}, n\geq 2.$ Let $W_{n}^{[i]}=M_{i}\oplus N_{i}$, $i\ge 0$, where $M_{i}$, $N_{i}$ are irreducible submodules of the $W_{n}^{[0]}$-module $W_{n}^{[i]}$ from Theorem \ref{Th1}. Then
	\begin{enumerate}
	\item $\dim W_n^{[i]}=n\binom{n+i}{i+1}, \  \dim N_i=\binom{n+i-1}{i}, \   \dim M_i=\frac{(n+i+1)(n+i-1)!}{(i+1)!(n-2)!};$
	\item If  $S, T$ are submodules of the $W_n^{[0]}$-module $W_n,$ then $[S,  T] $ is also a submodule of $W_n;$
		\item $[M_{i}, M_{j}] = M_{i+j}$ for all $i,j\ge 0$;
		\item $[N_{i}, N_{j}] = N_{i+j}$ for all $i,j\ge 0$, $i\ne j$; $[N_{i}, N_{i}] = 0$, for all $i\ge 0$;
		\item $[M_{i}, N_{j}] = W_{n}^{[i+j]}$ for all $i,j\ge 0$, except the cases:
		(a) $i=j=0$, where $[M_{0}, N_{0}] = 0$, 
		
		(b) $i= 0$, $j= 1,2,\dots$, where $[M_{0}, N_{j}] = N_{j}$,
		
		(c) $ i\geq 1 , j=0,  $ where $ [M_i, N_0]=M_i.$
			\item $[W_n^{[-1]}, N_{j}] = W_{n}^{[j-1]}$ for all $j\ge 0$
	\end{enumerate}
 \end{proposition}
\begin{proof}
	
	\begin{enumerate}
			\item Follows immediately.
				\item  Obvious.
\item 
		 To prove this part of the lemma it is enough to show (in  view of Theorem 1 and part 2 of this lemma)  that $[M_{i}, M_{j}]\ne 0$ for $i,j\ge 0$. Take the elements $x_{1}^{i+1}\frac{\partial}{\partial x_{2}}\in M_{i}$, $x_{2}^{j+1}\frac{\partial}{\partial x_{1}}\in M_{j}$. Then 
		$$[x_{1}^{i+1}\frac{\partial}{\partial x_{2}}, x_{2}^{j+1}\frac{\partial}{\partial x_{1}}]= (j+1)x_{1}^{i+1}x_{2}^{j}\frac{\partial}{\partial x_{1}}-(i+1)x_{1}^{i}x_{2}^{j+1}\frac{\partial}{\partial x_{2}}\ne 0.$$
		So, $[M_{i}, M_{j}]\ne 0$ for $i,j\ge 0$ and therefore $[M_{i}, M_{j}]= M_{i+j}$.
\item 
		As in the proof of the part 3 of the lemma it is enough to show that $[N_{i}, N_{j}]\ne 0$ for $i\ne j$, $i,j\ge 0$ (obviously $[N_i, N_j]\subseteq N_{i+j}$). Take the elements $x_{1}^{i}E_{n}\in N_{i}$, $x_{2}^{j}E_{n}\in N_{j}$. Then
		$$[x_{1}^{i}E_{n}, x_{2}^{j}E_{n}]= (j x_{1}^{i}x_{2}^{j}-i x_{1}^{i}x_{2}^{j})E_{n}\ne 0$$
		for $i,j\ge 0$, $i\ne j$.
		So, $[N_{i}, N_{j}]= N_{i+j}$ for all $i,j\ge 0$, $i\ne j$.
		 The equality $[N_{i}, N_{i}]= 0$ for $i\ge 0$ can be immediately checked.
		
\item 
		If $i= 0$, then $[M_{0}, N_{j}]= N_{j}$ for $j\ge 1$ because $N_{j}$ is an irreducible $W_{n}^{[0]}$-module. Of course, $[M_{0}, N_{0}]= 0$ since $N_{0}= E_{n}\K$.
		
		Let $i\ge 1$, $j\ge 1$. Take the elements $x_{2}^{i+1}\frac{\partial}{\partial x_{1}}\in M_{i}$ and $x_{1}^{j}E_{n}\in N_{j}$. Let us show that the product $T=[x_{2}^{i+1}\frac{\partial}{\partial x_{1}}, x_{1}^{j}E_{n}]$ lies neither in $M_{i+j}$ nor in $N_{i+j}$. Indeed, this product has the form $T=j x_{1}^{j-1}x_{2}^{i+1}E_{n}-i x_{1}^{j}x_{2}^{i+1}\frac{\partial}{\partial x_{1}}$. Since $i\ge 1$, $j\ge 1$ $T$ has nonzero divergence ($\Div (T=j(n+i+j-1)x_1^{j-1}x_2^{i+1}) $) and is  $T$  not proportional to the Euler derivation $E_{n}$. Thus, $T\notin M_{i+1}$ and $T\notin N_{i+1}$. The latter means that the $W_{n}^{[0]}$-submodule $[M_{i}, N_{j}]$ is nonzero and equals neither  $M_{i+j}$ nor  $N_{i+j}$. By Lemma \ref{lm:l3} we have $[M_{i}, N_{j}]= W_{n}^{[i+j]}$. Let us look at the remaining case $[M_i, N_0]. $ Since $N_0=E_n\mathbb K ,$ where $ E_n$ is the Euler derivation which acts scalarly on $M_i$ (in view of Lemma \ref{lm:l2}) we see that $ [M_i, N_0]=M_i.$
		\item It is obvious that $[W_n^{[-1]}, N_{j}]\subset W_n^{[j-1]}.$ Therefore it is sufficient to show that the $W_n^{[0]}$-module $[W_n^{[-1]}, N_{j}]$ is nonzero and does not coincide with neither $M_{j-1}$ nor $N_{j-1}.$ Take a nonzero element $fE_n\in N_j.$ Then $[\frac{\partial}{\partial x_k}, fE_n]=\frac{\partial f}{\partial x_k}E_n+f\frac{\partial}{\partial x_k}. $ The latter equality implies that $[W_n^{[-1]}, N_{0}]=W_n^{[-1]}.$ Let $j>0$. Then $\deg f\geq 1$ and at least one partial derivative $\frac{\partial f}{\partial x_k} $ is nonzero. As the product $ f\frac{\partial}{\partial x_k}$ does not belong to $N_{j-1}$, we see that $ [W_n^{[-1]}, N_{j}]=W_n^{[j-1]}$ by Lemma \ref{lm:l3}.
		
	\end{enumerate}
	
	The proof is complete.
\end{proof}

\begin{corollary}
	Let  $W^{[i]}_n$ and $W^{[j]}_n$ be homogeneous components  of the standard grading $W_n = W_{n}(\K)= \underset{i\ge-1}{\bigoplus}W_{n}^{[i]}, n\geq 2.$ Then  the next equalities are satisfied: 
	$$[W^{[i]}_n, W^{[j]}_n]= W^{[i+j]}_n, i, j\geq 0,$$ except the case when  $i=j=0$, where $[W_{n}^{[0]}, W_{n}^{[0]}]=M_{0}\simeq \sl_{n}(\K )$. Obviously $[W_n^{[-1]}, W_n^{[-1]}]=0$ since $[\frac{\partial}{\partial x_i}, \frac{\partial}{\partial x_j}]=0, i,j=1,\ldots ,n. $
\end{corollary}
\begin{proof}
	In view of Lemma \ref{lm:l3} and Proposition \ref{products} we need only to show that 
	$[W_{n}^{[-1]}, W_{n}^{[m]}]\not = 0$, $m\ge 0$. But this inequality can be checked immediately. 
\end{proof}

\begin{proposition}
Let $W_n^{[i]}=M_i+N_i$  be the decomposition of the $\sl_n(\mathbb K)$-module $ W_n^{[i]}$  into direct sum of irreducible submodules (as in Theorem \ref{Th1}). The modules $M_i, M_j, i\not = j $ are non-isomorphic, the same is true for modules $N_i, N_j, i\not = j.$  The modules $M_i, N_j  $ are non-isomorphic except the case of submodules $M_i, N_{i+2} $ of the $\sl_2(\mathbb K)$-module $ W_2$ when $M_i\simeq N_{i+2}. $ These $\sl_2(\mathbb K)$-modules are also  $\gl_2(\mathbb K)$-modules, but they are non-isomorphic as $\gl_2(\mathbb K)$-submodules of $W_2. $
\end{proposition}
\begin{proof}
The modules $M_i$ and $M_j, \ i\not = j$	are non-isomorphic because of different dimensions; the same is true for $N_i$ and $N_j, \  i\not = j.$ Consider the modules $M_i$ and $N_j, \ 1\leq i, j\leq n .$ 	
	It follows from the general theory of modules over a semisimple Lie algebra $\g $ that irreducible finite dimensional $\g$-modules are uniquely defined (up to isomorphism) by their maximal weights (see, for example, \cite{Humphreys}, Ch.VI). Thus, the modules $M_k$  and $ N_m  $  are isomorphic if and only if they have equal weights.  We have from the proof of Theorem \ref{Th1} that the module $N_m $ has the weight 
	$$\lambda ^{(1)}(x_i\frac{\partial}{\partial x_i} - x_n\frac{\partial}{\partial x_n})=m\delta_{i1}, i=1,\ldots , n-1$$ and the module  $M_k$  has the weight 
	$$\lambda ^{(2)}(x_i\frac{\partial}{\partial x_i} - x_n\frac{\partial}{\partial x_n})=1+(k+1)\delta _{i1}, i=1,\ldots , n-1.$$
	So, $M_k\simeq N_m $ if and only if $m\delta _{i1}=1+(k+1)\delta _{i1}, i=1, \ldots , n-1. $ If $n>2$, then  for $  i=2$ we have $0=m\delta _{12}=1+(k+1)\delta _{12}=1, $  which is impossible. Therefore $n=2, m=k+2.$ One can easily check that the irreducible  $sl_2(\mathbb K)$-modules  $M_k $ and $N_{k+2} $ have the same dimension (over $\mathbb K$), so they are isomorphic.
	The proof is complete.
\end{proof}
The subalgebra $W_{n}^{[0]}$ of the Lie algebra $W_{n}(\K)$ is contained in the subalgebra $W_{n}^{[-1]}\oplus W_{n}^{[0]}$ which is isomorphic to the linear affine Lie algebra $\aff_{n}(\K)$. This Lie algebra is the semidirect sum of the Lie algebra $\gl (V)$ and the abelian Lie algebra $V$, $\dim_{\K} V= n$, with the natural action of $\gl (V)$ on $V$. Note that the Lie algebra $W_{n}^{[-1]}\oplus W_{n}^{[0]}$ is not maximal in $W_{n}(\K)$ because it is contained in the Lie subalgebra of $W_{n}(\K)$ consisting of all the derivations with constant divergence.

Nevertheless,  the subalgebra $W_{n}^{[-1]}\oplus W_{n}^{[0]}$ has some properties similar to those of maximal subalgebras of $W_{n}(\K)$.

\begin{theorem}\label{Th2}
	Let $L=W_{n}^{[-1]}\oplus W_{n}^{[0]}$ be the subalgebra of $W_n,  \ n\geq 2$ consisting of all the derivations $D= \sum_{i=1}^{n}f_{i}\frac{\partial}{\partial x_{i}}$ with $\deg f_{i}\le 1$, $i= 1,\dots , n$. An element $D= D_{-1}+D_{0}+D_{1}+\cdots +D_{k}, D_i\in W_n^{[i]}$, $k\ge 1$ with  $D_{k}\ne 0$,  generates together with $L$ the entire Lie algebra $W_{n}$ if and only if $\Div D\ne {\rm const}$ and either $k\ge 2$ or $k=1$ and $D_{1}$ is not proportional to the Euler derivation $E_{n}$.
\end{theorem}
\begin{proof}
	Let a derivation $D\in W_{n},$
	$$D= D_{-1}+D_{0}+D_{1}+\cdots +D_{k}, \   k\ge 1, \ D_{k}\ne 0,$$ satisfy all the conditions of the theorem. Denote by the $U=\langle L, D\rangle$ the subalgebra of $W_{n}$ generated by $L$ and $D$. Let us show that $U= W_{n}$. Since the Euler derivation $E_{n}$ belongs to $L$ the subalgebra $U$ is graded: $U= \underset{i\ge-1}{\bigoplus}U_{i}$, where $U_{i}= U\cap  W_{n}^{[i]}$.	
		Indeed, for any element 
	$$S\in U, S= S_{-1}+S_{0}+S_{1}+\cdots +S_{m}, S_{i}\in W_{n}^{[i]}$$ it holds (by Lemma \ref{lm:l2}) the equality  
	$$[E_{n}, S]= -S_{-1}+0 S_{0}+S_{1}+\cdots +m S_{m}.$$ By induction on $m$ one can prove that $S_{i}\in U$, $i= -1,0,1,\dots, m$, so $U_{i}=W_{n}^{[i]}\cap  U$. Note that the $\K$-subspace $U_{i}$, $i\ge -1$, is a submodule of the $W_{n}^{[0]}$-module $W_{n}^{[i]}$ because $W_{n}^{[0]}\subset U$. Therefore, for any $i\ge 0,$ the submodule $U_{i}$ equals either $W_{n}^{[i]}$ or $M_{i}$, or $N_{i}$ by Lemma \ref{lm:l3}. Since $D_{k}\ne 0$ in the expression $D= D_{-1}+D_{0}+D_{1}+\cdots +D_{k}$, the submodule $U_{k}$ is nonzero.
	
	First, let $k=1$, i.e. $D= D_{-1}+D_{0}+D_{1}$. Since $D_{1}$ is not proportional to $E_{n}$ we have $U_{1}\ne N_{1}$.
	The condition $\Div D_{1}\ne const$ implies that $U_{1}\ne M_{1}$. Therefore $U_{1}= W_{n}^{[1]}.$ But then $W_{n}^{[2]}=[W_{n}^{[1]}, W_{n}^{[1]}]\subseteq U, \ldots , W_n^{[i]}\subseteq U,i\geq 2.  $  The latter means that  $U= W_{n}(\K).$
		
	Now,  let $k\ge 2$. If $U_{k}= W_{n}^{[k]}$ then taking into account the inclusion $W_{n}^{[-1]}\subset U,$ we see that $W_{n}^{[k-1]}\subset U,\dots , W_{n}^{[1]}\subset U$. But then $W_{n}^{[i]}\subset U$ for $i\ge 1$ and $U= W_{n}(\K)$. Further, if $U_{k}= N_{k}$ then $[W_{n}^{[-1]}, N_{k}]= W_{n}^{[k-1]}$ by Proposition \ref{products}. One can show in similar way that $U= W_{n}(\K)$. Let now $U_{k}= M_{k}$. Since $\Div D\ne const$ we have $\Div D_{s}\ne const$ for some $1\le s \le k$. But then $U_{s}\ne 0$, $U_{s}\ne M_{s}$. If $U_{s}= W_{n}^{[s]}$ then we can show as above that $U= W_{n}(\K)$. Let $U_{s}= N_{s}$. Then, by Proposition \ref{products},  $[U_{k}, U_{s}]= [M_{k}, N_{s}]= W_{n}^{[k+s]}$. Repeating the previous considerations  one can show that $U= W_{n}$.
		
	Now, let $L$ and $D$ generate the Lie algebra $W_{n}$. By Lemma \ref{lm:l1}, all the derivations with constant divergence form a proper subalgebra $W_c$ of $W_{n}(\K)$. Since $L\subseteq W_c$ we see that $\Div D\ne const$. Write the derivation $D$ in the form $$D= D_{-1}+D_{0}+D_{1}+\cdots +D_{k}, \ D_{i}\in W_{n}^{[i]}.$$ Since $\Div D\ne const$ we see that $k\ge 1$. If $k=1$ and $D_{1}= f E_{n}$ for some $f\in \K[x_{1},\dots , x_{n}]$ then $D_{1}\in N_{1}$. But $L+N_{1}$ is a proper subalgebra of $W_{n}(\K)$. The latter contradicts our assumption on $D$. Therefore $D_{1}$ is non-proportional to the Euler derivation $E_{n}$. 
		The proof is complete.
\end{proof}

Using Theorem 2 one can easily prove the next statement (see \cite{Bavula}, Proposition 2.21).

\begin{corollary}
	The set $W_c$ of all derivations from the Lie algebra $ W_{n}, n\geq 2$ with constant divergence is a maximal subalgebra  of the Lie algebra $ W_{n}(\mathbb K)$. 
\end{corollary}


%
\end{document}